\documentclass[12pt]{article}
\usepackage{graphicx}
\usepackage{amssymb}
\usepackage{latexsym, amssymb, amsfonts, amsmath, amsthm, graphics, graphicx, subfigure, bbm, stmaryrd,xcolor,ulem}
\usepackage{epstopdf}

 \usepackage{authblk}

\usepackage[margin=2cm]{geometry}

\usepackage{latexsym, amssymb, amsfonts, amsmath, amsthm, graphics, graphicx, subfigure, bbm, stmaryrd}
\usepackage{epstopdf}

\allowdisplaybreaks[4]

\topmargin by -\headsep \textheight 8.9in \oddsidemargin 0pt
\evensidemargin \oddsidemargin \marginparwidth 0.5in \textwidth
6.5in

\newcommand{\ZZ}{{\mathbb Z}}

\newcommand{\EE }{{\mathbb E}}

\newcommand{\PP}{{\mathbb P}}

\newcommand{\dsum}{\displaystyle\sum}

\newcommand{\e}{\mathbf{e}}

\newtheorem{theorem}{Theorem}[section]

\newtheorem{corollary}[theorem]{Corollary}
\newtheorem{lemma}[theorem]{Lemma}

\newtheorem{remark}[theorem]{Remark}

\numberwithin{equation}{section}

\makeatletter
\renewcommand\section{\@startsection {section}{1}{\z@}%
    {-3.5ex \@plus -1ex \@minus -.2ex}%
    {2.3ex \@plus.2ex}%
    {\normalfont\fontsize{18}{19}\bfseries}}
\makeatother

\makeatletter
\renewcommand\subsection{\@startsection {subsection}{1}{\z@}%
    {-1.5ex \@plus -1ex \@minus -.2ex}%
    {1.3ex \@plus.2ex}%
    {\normalfont\fontsize{13}{14}\bfseries}}
\makeatother

\makeatletter
\newcommand\xleftrightarrow[2][]{%
  \ext@arrow 9999{\longleftrightarrowfill@}{#1}{#2}}
\newcommand\longleftrightarrowfill@{%
  \arrowfill@\leftarrow\relbar\rightarrow}
\makeatother

\title{\bf Percolation of finite clusters and shielded paths}

\author{Bounghun Bock\thanks{School of Mathematics, Georgia Institute of Technology, Atlanta, GA 30332 USA}
\hspace{.5cm} Michael Damron\thanks{School of Mathematics, Georgia Institute of Technology, Atlanta, GA 30332 USA}
\and C. M. Newman\thanks{Courant Institute of Mathematical Sciences, New York, NY
10012 USA, and NYU-ECNU Institute of Mathematical Sciences at NYU
Shanghai, 3663 Zhongshan Road North, Shanghai 200062, China}
\hspace{.5cm} Vladas Sidoravicius\thanks{NYU-ECNU Institute of Mathematical Sciences at NYU
Shanghai, 3663 Zhongshan Road North, Shanghai 200062, China}}

\begin{document}

\baselineskip20pt

\maketitle

\begin{abstract}
In independent bond percolation on $\mathbb{Z}^d$ with parameter $p$, if one removes the vertices of the infinite cluster (and incident edges), for which values of $p$ does the remaining graph contain an infinite connected component? Grimmett-Holroyd-Kozma used the triangle condition to show that for $d \geq 19$, the set of such $p$ contains values strictly larger than the percolation threshold $p_c$. With the work of Fitzner-van der Hofstad, this has been reduced to $d \geq 11$. We improve this result by showing that for $d \geq 10$ and some $p>p_c$, there are infinite paths consisting of ``shielded'' vertices --- vertices all whose adjacent edges are closed --- which must be in the complement of the infinite cluster. Using values of $p_c$ obtained from computer simulations, this bound can be reduced to $d \geq 7$. Our methods are elementary and do not require the triangle condition.
\end{abstract}

\section{Introduction}

In bond percolation, we declare each edge $e$ in the set $\mathcal{E}^d$ of nearest-neighbor edges of $\mathbb{Z}^d$ to be open or closed with probability $p$ or $1-p$, independently of each other. The resulting product measure $\mathbb{P}_p$ on the space $\{0,1\}^{\mathcal{E}^d}$ then has $\mathbb{P}_p(\omega(e) = 1) = p = 1-\mathbb{P}_p(\omega(e)=0)$ for all $e$, and edges $e$ with $\omega(e)=1$ (respectively, 0) we call open (respectively, closed). The main object of study in bond percolation is the connectivity of open clusters, in particular whether there are infinite open clusters. Specifically, two vertices $x$ and $y$ are said to be connected by an open path (written $x\to y$) if there is a path (an alternating sequence of vertices and edges $v_0,e_0, v_1, \dots, e_{n-1},v_n$ such that $e_i$ has endpoints $v_i$ and $v_{i+1}$) from $x$ to $y$ whose edges are open. The (open) clusters in an outcome $\omega$ are the connected components of $\mathbb{Z}^d$ under the equivalence relation $x \to y$. If one defines
\[
p_c = p_c(d) = \inf\left\{ p \in [0,1] : \mathbb{P}_p(\text{there is an infinite cluster})>0\right\},
\]
then one can show (see for example \cite[Theorem~1.10]{grimmett_percolation}) that $p_c \in (0,1)$ for all $d \geq 2$ and $p_c(d) \sim \frac{1}{2d}$ as $d \to \infty$ (see \cite{MR1136447, MR1064563}).

A natural question for $p>p_c$ is to determine the geometric properties of infinite clusters. It is known (see for example \cite[Theorem~8.1]{grimmett_percolation}) that a.s., there is a unique infinite cluster and its asymptotic density is $\theta(p) = \mathbb{P}_p(0 \text{ is in an infinite cluster})>0$. In this paper, following Grimmett-Holroyd-Kozma \cite{MR3177869}, we study the complement of the infinite cluster. Let $X$ be the subgraph of $\mathbb{Z}^d$ obtained after removing all vertices in the infinite cluster. The complementary critical value, $p_{fin}$, is defined as
\[
p_{fin} = p_{fin}(d) = \sup\{ p \in [0,1] : \mathbb{P}_p(X \text{ has an infinite connected component})>0\}.
\]
In dimension $d=2$, it is known that $p_c = 1/2$ \cite{MR575895} and that for each $p>p_c$, the infinite cluster contains infinitely many circuits (paths whose initial and final points coincide) around the origin. This implies that $p_{fin}(2) \leq p_c(2)$. Because the definition of $p_{fin}(d)$ implies
\begin{equation}\label{eq: p_fin_trivial_bound}
p_{fin}(d) \geq p_c(d) \text{ for all }d,
\end{equation}
we obtain $p_{fin}(2) = 1/2$.

Due to \eqref{eq: p_fin_trivial_bound}, one is led to ask: for which $d$ do we have $p_{fin}(d) > p_c(d)$? It is natural to believe that this is true for large $d$ because $\theta(p_c)=0$ (see for example \cite[Section~10.3]{grimmett_percolation}) and so for $p = p_c +\epsilon$ and $\epsilon>0$ small, one expects an infinite cluster with small asymptotic density whose removal is likely to leave much of $\mathbb{Z}^d$ intact. The inequality $p_c(d)  < p_{fin}(d)$ for $d \geq 19$ was proved by Grimmett-Holroyd-Kozma in \cite{MR3177869} using the triangle condition \cite{MR762034}. Later, Fitzner-van der Hofstad verified the triangle condition for $d \geq 11$ \cite{MR3646069}, so
\begin{equation}\label{eq: GHK_result}
p_c(d) <p_{fin}(d) \text{ for } d \geq 11. 
\end{equation}

We will develop a different approach to $p_{fin}$ involving ``shielded percolation.'' We call the vertex $x$ shielded if all edges incident to $x$ are closed. A path whose vertices are shielded is called a shielded path. We define the shielded critical probability as
\[
p_{shield} := \sup \{ p \in [0,1] : \PP_p(\text{$\exists$ an infinite shielded path}) >0 \}.
\]
There a.s. exists an infinite shielded path if $p<p_{shield}$. Furthermore, by the definition of the critical shielded probability, if $p < p_{shield}$, then there exists an infinite connected component in $X$. Thus, for any $d$,
\begin{equation} 
\begin{split}
p_{shield}(d)\leq p_{fin}(d).
\end{split}
\end{equation}
By giving lower bounds on $p_{shield}$, we therefore obtain them for $p_{fin}$. Our goals in this paper are two-fold. One is to improve the Grimmett-Holroyd-Kozma result \eqref{eq: GHK_result} from $d \geq 11$ to $d \geq 10$, but perhaps more importantly, our arguments, using shielded percolation, are much more elementary and not based on the triangle condition. Furthermore, using numerical values of $p_c$ from \cite{MR1976824, 2018arXiv180608067M}, we will also verify that \eqref{eq: GHK_result} should hold for all $d \geq 7$. (See the comments below Corollary~\ref{p_{shield >p_c} by numerical} and numerical estimates in the appendix.)

The idea for proving lower bounds for $p_{shield}$ is adapted from Cox-Durrett \cite{MR684285}, in their study on the asymptotics of the threshold for oriented percolation. (In that paper, the idea is attributed to Kesten.) One shows that for certain values of $p$, the number of open oriented paths from 0 to distant hyperplanes has uniformly positive mean, and suitably bounded second moment. The Paley-Zygmund inequality then implies that there are oriented infinite clusters for such $p$. In running a version of this argument for shielded paths, we obtain the existence of infinite oriented shielded paths for certain values of $p$. Because the oriented shielded value is smaller than $p_{shield}$, it is conceivable that more sophisticated lower bounds for $p_{shield}$ would reduce the dimensions (10 and 7) in our results.

\subsection{Main results}

We begin with an explicit upper bound for $p_{shield}$. Let $\lambda(d)$ be the connective constant for vertex self-avoiding walks on $\mathbb{Z}^d$. It is defined (by sub-multiplicativity) as 
\[
\lambda(d) = \lim_{n \to\infty} \left( \#\{\text{vertex self-avoiding paths with } n \text{ vertices, started at }0\}\right)^{1/n}.
\]

\begin{theorem} \label{upperbound of shield} For any $d \geq 1$,
\[
p_{shield} (d)\leq   1- \lambda(d)^{-\frac{1}{2d-1}}.
\]
\end{theorem}
The actual bound obtained from the proof of Theorem~\ref{upperbound of shield} is $p_{shield} \leq 1-\tilde{\lambda}(d)^{-1/(2d-1)}$. Here, $\tilde{\lambda}(d)$ is the connective constant for vertex self-avoiding walks $\Gamma = (x_0, x_1, \dots)$ with the property that $\|x_k-x_\ell\|_1 = 1$ if and only if $|\ell-k|=1$. This connective constant is strictly below $\lambda(d)$, so the inequality in the theorem is strict. 

Using the elementary bound $\lambda(d) \leq 2d-1$, Theorem~\ref{upperbound of shield} implies
\begin{equation}\label{eq: thm_1_consequence}
p_{shield}(d) \leq 1- \left( \frac{1}{2d-1}\right)^{\frac{1}{2d-1}}.
\end{equation}
Therefore $p_{shield}(2) \leq 1-\left(\frac{1}{3}\right)^{\frac{1}{3}} \sim 0.306... < \frac{1}{2}$, and we obtain
\[
p_{shield}(2) < p_{fin}(2) = p_c(2).
\] 
(For $d=3$, we obtain $p_{shield}(3) \leq 1-\left(\frac{1}{5}\right)^{\frac{1}{5}} \sim 0.275...$, which is larger than $p_c(3) \sim 0.248...$.) In contrast, the next result implies that $p_{shield}(d) > p_c(d)$ for large $d$. 

Write $\e_i$ for the $i$-th standard basis vector, and let $(X_n), (X_n')$ be i.i.d.~with $\PP(X_n=\e_i)=\PP(X_n'=\e_i)=\frac{1}{d}$ for $1\leq i \leq d$. $S_n, S_n'$ are defined as the sum of the first $n$ terms of $(X_n)$ and $(X_n')$ respectively with $S_0=S_0'=0$. Define the probability
\[
p_2 = \mathbb{P}(\|S_n - S_n'\|_1=2 \text{ for some } n \geq 2 \mid \|S_1 - S_1'\|_1=2).
\]
The condition $d\geq 4$ in the following theorem is needed only to guarantee that $p_2<1$.
 
\begin{theorem}\label{lowerbound of shield}  
Suppose that $d \geq 4$ and that $p$ satisfies the conditions
\begin{enumerate}
\item $p < 1-\left( \frac{1}{d} \right)^{\frac{1}{2d-1}}$ and
\item $\frac{1}{1-p} \bigg( p_2 - \frac{1}{d^2} + \frac{1}{d} \left( 1-\frac{1}{d} \right) (d(1-p)^{2d-1}-1)^{-1} \bigg) < 1$.
\end{enumerate}
Then $p_{shield}(d) \geq p$.
\end{theorem}
The previous result states that $p_{shield}$ can be bounded in terms of the probability $p_2$. It is difficult to find the exact value of $p_2$, but at least we can calculate bounds for it. As a result of above theorems, we get the following corollaries. Write $a_n \sim b_n$ for real sequences $(a_n)$ and $(b_n)$ if $a_n/b_n \to 1$ as $n\to\infty$.

\begin{corollary} \label{corollary p_{shield}}
\begin{equation} 
\begin{split}
p_{shield}(d) \sim \frac{\log d}{2d}  \; \text{as $d\rightarrow \infty$}.
\end{split}
\end{equation}
\end{corollary}

\begin{remark}
Corollary~\ref{corollary p_{shield}} implies that
\[
\liminf_{d \to \infty} \frac{p_{fin}(d)}{\frac{\log d}{2d}} \geq 1.
\]
It would be interesting to have asymptotic upper bounds for $p_{fin}(d)$. Is $\frac{\log d}{2d}$ the correct order of $p_{fin}$, as it is \cite{MR3177869} on the $2d$-regular tree?
\end{remark}

\begin{corollary} \label{p_{shield >p_c} by numerical} For $d \geq 10$, 
\begin{equation} 
\begin{split}
p_c(d) < p_{shield}(d) \leq p_{fin}(d).
\end{split}
\end{equation}
\end{corollary}

\begin{remark}
For $p<p_{shield}$, it is immediate that the graph induced by the set of shielded vertices has a percolation threshold strictly smaller than 1. This is of interest in the context of frozen percolation models, where edges are opened according to independent exponential clocks. Soon after an infinite component emerges, all edges adjacent to the component are ``frozen'' (do not change their state anymore). If this freezing occurs at parameter $p \in (p_c,p_{shield})$, then a new infinite component will emerge at a later time.
\end{remark}

The $d \geq 10$ bound does not appear to be related to mean-field behavior, as the bounds of \cite{MR3177869} are. If numerical values of $p_c(d)$ from \cite{MR1976824} or \cite{2018arXiv180608067M} are used, we can improve the dimension in Corollary~\ref{p_{shield >p_c} by numerical} to $d \geq 7$. This is shown in Table \ref{numerical result} in the appendix. The $d \geq 7$ result is only partially rigorous, since it uses estimates for $p_c$ obtained from computer simulations to find dimensions $d$ for which the inequalities of Theorem~\ref{lowerbound of shield} hold. We should mention that the method of \cite{MR3177869}, combined with non-rigorous estimates for the one-arm exponent would allow to show that $p_{fin}(d) > p_c(d)$ for $d \geq 6$.

 \subsection{Outline of the paper}

In the next section, we give a short proof (relying only on the fact that $p_c(d) \sim 1/(2d)$ as $d \to\infty$ --- this fact is not needed in any of our other results) that $p_{shield}(d) > p_c(d)$ for $d$ large enough. The proof we give would be difficult to make quantitative, since it uses (far from optimal) estimates from $1$-dependent percolation. We present it because it gives a simple explanation for the inequality in high dimension. In Section~\ref{sec: proofs_of_theorems}, we prove Theorems~\ref{upperbound of shield} and \ref{lowerbound of shield}. In Section~\ref{sec: proofs_of_corollaries}, we prove Corollaries~\ref{corollary p_{shield}} and \ref{p_{shield >p_c} by numerical}. Finally, in the appendix, we explain how we show numerically that for $d \geq 7$, one has $p_{shield}(d) > p_c(d)$.

\section{Short proof of $p_{shield}(d) > p_c(d)$ for large $d$}

In this section we give a short proof that $p_{shield}(d)$ (and therefore $p_{fin}(d)$) is larger than $p_c(d)$ if $d$ is large. The idea is that since $p_c(d) \sim 1/(2d)$, the probability that a vertex is shielded at parameter $p_c(d)$ is $(1-p_c(d))^{2d} \to e^{-1}$ as $d\to\infty$. Since $p_c^{site} \to 0$, one can reasonably expect the shielded sites to percolate. Since the shielded sites are not independent, but are rather 1-dependent, we will need to work a little harder.

Let $a>1$ and fix $d_\ast$ so that 
\begin{equation}\label{eq: d_ast_choice}
p_c^{site}(d_\ast) < e^{-a}.
\end{equation} 
(This is possible since $p_c^{site}(d) \to 0$ as $d \to\infty$.) For $d\geq d_\ast$, set 
\[
\mathbb{Z}_\ast^d = \{x \in \mathbb{Z}^d : x \cdot \e_i = 0 \text{ for } i = d_{\ast}+1, \dots, d\}.
\]
We say that a vertex $x \in \mathbb{Z}_\ast^d$ is partially shielded if all edges of the form $\{x,x \pm \e_i\}$ are closed for $i=d_\ast+1, \dots, d$. Note that the partially shielded vertices form an independent site percolation process on $\mathbb{Z}^d_\ast$ with parameter $(1-p)^{2(d-d_\ast)}$. Set $p = p(d) =\frac{a}{2d}$ so that, because $p_c(d) \sim \frac{1}{2d}$, we have $p > p_c$ for large $d$. Furthermore, for any $x \in \mathbb{Z}^d_\ast$,
\[
\mathbb{P}_p(x \text{ is partially shielded}) = \left( 1-\frac{a}{2d}\right)^{2(d-d_\ast)} \to e^{-a} \text{ as } d \to\infty.
\]

For $x \in \mathbb{Z}^d_\ast$, we define $Y_x$ to be the indicator of the event that all edges of the form $\{x,x\pm \e_i\}$ are closed for $i=1, \dots, d_\ast$. Then the $Y_x$'s form a 1-dependent site percolation process on $\mathbb{Z}^d_\ast$ (independent of the process of partially shielded vertices) such that for any $x \in \mathbb{Z}^d_\ast$,
\[
\mathbb{P}_p(Y_x = 1) = \left( 1-\frac{a}{2d}\right)^{2d_\ast} \to 1 \text{ as } d \to\infty.
\]
Therefore the result of Liggett-Schonmann-Stacey (presented in \cite[Theorem~7.65]{grimmett_percolation}) implies that $(Y_x)$ is stochastically bounded below by an independent site percolation process $(Z_x)$ with $\mathbb{P}(Z_x =1) \to 1$ as $d \to\infty$. We will assume that the variables $Z_x$ are coupled with the original percolation process so that if $Z_x = 1$, then $Y_x=1$ and that the $Z_x$'s are independent of the process of partially shielded vertices.

Call $x \in \mathbb{Z}^d_\ast$ green if $x$ is partially shielded and $Z_x = 1$. Then the set of shielded vertices in $\mathbb{Z}^d_\ast$ contains the set of green vertices. Since
\[
\mathbb{P}_p(x \text{ is green}) = \left( 1-\frac{a}{2d}\right)^{2(d-d_\ast)} \mathbb{P}(Z_x=1) \to e^{-a} \text{ as } d \to\infty,
\]
inequality \eqref{eq: d_ast_choice} implies that for $d$ large, this probability is $>p_c^{site}(d_\ast)$. Because the green sites form an independent site percolation process on $\mathbb{Z}^d_\ast$, one has 
\[
\mathbb{P}_p(\text{there is an infinite component of green vertices})>0 \text{ for large } d.
\]
This implies that for large $d$, one has $p_{shield}(d) \geq p = \frac{a}{2d} >p_c(d)$. 

Note that since this proof works for any $a>1$, it actually shows that $p_{shield}(d) /p_c(d) \to \infty$ as $d \to \infty$.

\section{Proofs of Theorems~\ref{upperbound of shield} and \ref{lowerbound of shield}}\label{sec: proofs_of_theorems}

\begin{proof} [Proof of Theorem \ref{upperbound of shield}]

Suppose $p < p_{shield}$; that is, there is a.s. an infinite shielded path $\Gamma$. This is a path, which we will take to be vertex self-avoiding, whose vertices are all shielded. By translation invariance, the probability that the origin is contained in such a path is positive. We will use a Peierls-type argument to show that $p$ cannot be too large.

To do this, we first extract from $\Gamma$ another vertex self-avoiding path with shielded vertices $0=y_0, y_1, \dots$ satisfying the following two conditions:
%
\begin{equation}\label{eq: sequence_property_1}
y_j \text{ is adjacent to }y_{j-1} \text{ for }j \geq 1
\end{equation}
 and 
\begin{equation}\label{eq: sequence_property_2}
y_j \text{ is not adjacent to any of }y_0, \dots, y_{j-2}\text{ for }j \geq 2.
\end{equation}
This can be accomplished using a loop-erasure procedure.

Let $\Xi_n$ be the set of sequences $0=y_0, \dots, y_n$ of distinct vertices with properties \eqref{eq: sequence_property_1} (for $j=1, \dots, n$) and \eqref{eq: sequence_property_2} (for $j=2, \dots, n$). Then the probability that any $\gamma \in \Xi_n$ is shielded is $q^{2d}(q^{2d-1})^n$, where $q=1-p$. Because $p<p_{shield}$, for each $n$,
\[
0 < \inf_m \mathbb{P}_p\left(\text{some } \gamma \in \Xi_m \text{ is shielded}\right) \leq \sum_{\gamma\in \Xi_n} q^{2d}(q^{2d-1})^n = q^{2d} \left( q^{n(2d-1)} \#\Xi_n \right).
\]
Because of property \eqref{eq: sequence_property_1}, each $\gamma \in \Xi_n$ is a vertex self-avoiding path with $n+1$ vertices, started at 0. The number of such paths equals $(\lambda(d)+o(1))^{n+1}$ as $n \to\infty$, so
\[
\text{if } p < p_{shield}(d), \text{ then } q^{2d}\left( q^{n(2d-1)} (\lambda(d)+o(1))^{n+1}\right) \text{ is bounded away from 0 as } n \to \infty.
\]
This implies $q^{2d-1} \lambda(d) \geq 1$, and so we find $p \leq 1-\lambda(d)^{-\frac{1}{2d-1}}$ for any $p$ satisfying $p<p_{shield}(d)$. This completes the proof.
\end{proof}

Next, we move to lower bounds for $p_{shield}$.

\begin{proof} [Proof of Theorem \ref{lowerbound of shield}] We use a version of the second moment method from oriented percolation in \cite{MR684285}. Let $R_n$ be the set of oriented paths from the origin to 
\begin{equation}\label{eq: H_n_def}
H_n :=\bigg\{ y \in \ZZ^d: \dsum_{i=1}^d (y\cdot \e_i) =n \bigg\}.
\end{equation}
Let $N_n$ be the (random) number of shielded paths in $R_n$. Then,
\begin{align}
\EE_p N_n &= \dsum_{\gamma \in R_n} \PP_p(\text{all sites in $\gamma$ are shielded}) = q^{2d} (dq^{2d-1})^n, \text{ and} \label{eq: first_moment}\\
\EE_p N_n^2 &= \dsum_{\gamma, \gamma' \in R_n} \PP_p(\text{all sites in $\gamma \cup \gamma'$ are shielded}),  \label{eq: new_second_moment}
\end{align}
where $q=1-p$. The object is now to find values of $p$ for which $\frac{\mathbb{E}_pN_n^2}{(\mathbb{E}_pN_n)^2}$ is bounded away from infinity. If we do this, then $\mathbb{P}_p(N_n \geq 1) = \mathbb{P}_p(N_n > 0) \geq \frac{(\mathbb{E}_pN_n)^2}{\mathbb{E}_pN_n^2}$ will be bounded away from zero, and there will be an infinite (oriented) shielded path with positive probability. For such values of $p$, then, we will have $p \leq p_{shield}$, and this produces a lower bound on $p_{shield}$. In other words,
\begin{equation}\label{eq: the_goal}
\text{if }  \sup_n \frac{\mathbb{E}_pN_n^2}{(\mathbb{E}_pN_n)^2} < \infty, \text{ then } p \leq p_{shield}.
\end{equation}

We now write the probability in the sum for $\mathbb{E}_p N_n^2$ as
\begin{equation}\label{eq: new_edge_probability}
\mathbb{P}_p(\text{all sites in $\gamma \cup \gamma'$ are shielded}) =q^{2(2d+n(2d-1)) - \#\mathcal{E}_n(\gamma,\gamma')},
\end{equation}
where $\mathcal{E}_n(\gamma,\gamma')$ is the set of edges with an endpoint in $\gamma$ and an endpoint in $\gamma'$ (this could be the same endpoint). We claim that
\begin{equation}\label{eq: new_edge_bound}
\#\mathcal{E}_n(\gamma,\gamma') \leq 2d + (2d-1)\#Z_n(\gamma,\gamma') + \#O_n(\gamma,\gamma'),
\end{equation}
where the vertices of $\gamma$ (in order) are written as $0 = x_0, \dots, x_n$, the vertices of $\gamma'$ (in order) are written as $0=x_0', \dots, x'_n$, and we define
\[
\begin{split}
Z_n(\gamma, \gamma') &= \{k=1, \dots, n: x_k = x_k'\}, \text{ and} \\
O_n(\gamma,\gamma') &= \{k=1, \dots, n: \|x_k-x_k'\|_1=2\}.
\end{split}
\]

The proof of \eqref{eq: new_edge_bound} is combinatorial in nature and will require us to introduce a few more sets. First partition the edges in $\mathcal{E}_n(\gamma,\gamma')$ into two subsets:
\[
\mathcal{E}_n^{(1)}(\gamma,\gamma') = \{e \in \mathcal{E}_n(\gamma,\gamma') : e \text{ is adjacent to some }x_k \text{ with } k \in Z_n(\gamma,\gamma') \cup \{0\}\}
\]
and
\[
\mathcal{E}_n^{(2)}(\gamma,\gamma') = \mathcal{E}_n(\gamma,\gamma') \setminus \mathcal{E}_n^{(1)}(\gamma,\gamma').
\]
Next, partition the numbers in $O_n(\gamma,\gamma')$ into two subsets:
\[
O_n^{(1)}(\gamma,\gamma') = \{k = 1, \dots, n-1 : k \in O_n(\gamma,\gamma'), k+1 \in Z_n(\gamma,\gamma')\}
\]
and
\[
O_n^{(2)}(\gamma,\gamma') = O_n(\gamma,\gamma') \setminus O_n^{(1)}(\gamma,\gamma').
\]
With these definitions, we will prove that
\begin{equation}\label{eq: step_1_comb}
\#\mathcal{E}_n^{(1)}(\gamma,\gamma') = 2d + (2d-1)\#Z_n(\gamma,\gamma') + \#O_n^{(1)}(\gamma,\gamma'),
\end{equation}
and
\begin{equation}\label{eq: step_2_comb}
\#\mathcal{E}_n^{(2)}(\gamma,\gamma') \leq \#O_n^{(2)}(\gamma,\gamma').
\end{equation}
If we sum these two displays, we obtain the desired inequality \eqref{eq: new_edge_bound}. Therefore, to show \eqref{eq: new_edge_bound}, it suffices to prove \eqref{eq: step_1_comb} and \eqref{eq: step_2_comb}.

First we prove \eqref{eq: step_1_comb}. We write $Z_n(\gamma,\gamma') \cup \{0\}$ as a union of (integer) intervals: it equals $[i_0,i_1] \cup [i_2, i_3] \cup \dots \cup [i_{2m},i_{2m+1}]$, where $i_0=0$, $m$ is a nonnegative integer, and $i_{2r} \geq i_{2r-1}+2$ for $r=1, \dots, m$. Then because the paths $\gamma,\gamma'$ in these intervals overlap entirely, we obtain
\begin{align*}
\#\mathcal{E}_n^{(1)}(\gamma,\gamma') &= \sum_{r=0}^m \#\{e \text{ adjacent to }\{x_{i_{2r}}, \dots, x_{i_{2r+1}}\}\} \\
&= (m+1) + (2d-1)\sum_{r=0}^m (i_{2r+1}-i_{2r}+1) \\
&= (m+1)+(2d-1)(\#Z_n(\gamma,\gamma')+1).
\end{align*}
However, the elements of $O_n^{(1)}(\gamma,\gamma')$ are exactly those that precede $i_2, i_4, \dots, i_{2m}$; that is, they are $i_2-1, \dots, i_{2m}-1$. Therefore $\#O_n^{(1)}(\gamma,\gamma') = m$, and this gives \eqref{eq: step_1_comb}.

To prove \eqref{eq: step_2_comb}, we construct an injective map from $\mathcal{E}_n^{(2)}(\gamma,\gamma')$ into $O_n^{(2)}(\gamma,\gamma')$. The existence of such a map immediately implies \eqref{eq: step_2_comb}. To define it, let $e \in \mathcal{E}_n^{(2)}(\gamma,\gamma')$. Since $e$ is adjacent to both $\gamma$ and $\gamma'$, but not to any vertices of $\gamma \cap \gamma'$, it must have one endpoint in $\gamma \setminus \gamma'$ and one in $\gamma' \setminus \gamma$. Therefore for some $k=1, \dots, n-1$, either $e = \{x_k,x_{k+1}'\}$ or $e = \{x_k',x_{k+1}\}$. In either case, we set $i(e) = k$.
\begin{enumerate}
\item[(a)] $i$ is injective. To see why, let $e,f \in \mathcal{E}_n^{(2)}(\gamma,\gamma')$ and suppose that $i(e) = i(f) = k$. Then if $x_k$ (or $x_k'$) is the lower endpoint of both $e$ and $f$, we must have $e=f=\{x_k,x'_{k+1}\}$ (or $\{x'_k,x_{k+1}\}$). Otherwise, up to a relabeling, $e = \{x_k,x_{k+1}'\}$ and $f = \{x_k',x_{k+1}\}$. Then $x_{k+1}$ and $x_{k+1}'$ are both adjacent to $x_k$ and $x_k'$. However $x_k \neq x_k'$ by the definition of $\mathcal{E}_n^{(2)}(\gamma,\gamma')$; thus there is at most one vertex in $H_{k+1}$ that is adjacent to both $x_k$ and $x_k'$. This gives $x_{k+1}=x_{k+1}'$, a contradiction, since $e \in \mathcal{E}_n^{(2)}(\gamma,\gamma')$.
\item[(b)] For all $e \in \mathcal{E}_n^{(2)}(\gamma,\gamma')$, one has $i(e) \in O_n^{(2)}(\gamma,\gamma')$. Indeed, write $e = \{x_k,x_{k+1}'\}$ (after a possible relabeling of $\gamma,\gamma'$) so that $i(e) = k$. Since $e \in \mathcal{E}_n^{(2)}(\gamma,\gamma')$, one has $x_k \neq x_k'$, but $\|x_k-x_k'\|_1 \leq \|x_k-x_{k+1}'\|_1 + \|x_k'-x_{k+1}'\|_1 = 2$. This implies that $\|x_k-x'_k\|_1=2$ and thus that $k \in O_n(\gamma,\gamma')$. A similar argument works for $k+1$: one has $x_{k+1} \neq x_{k+1}'$, but $\|x_{k+1}-x_{k+1}'\|_1 \leq \|x_{k+1}'-x_k\|_1 + \|x_{k+1}-x_k\|_2 = 2$. Therefore $k+1 \in O_n(\gamma,\gamma')$. We conclude that $k \in O_n^{(2)}(\gamma,\gamma')$.
\end{enumerate}

Now that we have shown \eqref{eq: new_edge_bound}, we combine it with \eqref{eq: new_edge_probability} and \eqref{eq: new_second_moment} to obtain
\begin{align}
\mathbb{E}_pN_n^2 &\leq \sum_{\gamma,\gamma' \in R_n} q^{2(2d+n(2d-1)) - 2d - (2d-1)\#Z_n(\gamma,\gamma') - \#O_n(\gamma,\gamma')} \nonumber \\
&= \frac{(\mathbb{E}_p N_n)^2}{q^{2d}d^{2n}} \sum_{\gamma,\gamma' \in R_n} q^{-\#O_n(\gamma,\gamma') - (2d-1)\#Z_n(\gamma,\gamma')}. \label{eq: second_moment}
\end{align}
We now represent $O_n( \gamma, \gamma')$ and $Z_n(\gamma, \gamma')$ using random walks. Let $(X_k), (X_k')$ be i.i.d.~sequences with
\begin{equation*}
\begin{split}
\PP (X_k=\e_i)=\PP(X_k'=\e_i)=\frac{1}{d} \;\; \text{for $1\leq i \leq d$}.
\end{split}
\end{equation*}
$S_n$ and $S_n'$ are defined as the sum of the first $n$ terms of $(X_k)$ and $(X_k')$ respectively with $S_0=S_0'=0$. Let
\begin{equation*}
\begin{split}
Z_n &= \{k=1,\dots, n : S_k = S_k'\}\\
O_n &= \{k =1,\dots, n : ||S_k- S_k' ||_1 =2\}.
\end{split}
\end{equation*}
Using these variables and \eqref{eq: second_moment}, we have the representation 
\[
\frac{\EE_p N_n^2}{(\EE_p N_n)^2} \leq   \frac{1}{q^{2d}} \EE \left[ q^{-\#O_n-(2d-1)\#Z_n}\right],
\]
so by the monotone convergence theorem,
\[
\frac{\mathbb{E}_pN_n^2}{(\mathbb{E}_pN_n)^2} \leq \frac{1}{q^{2d}} \EE q^{-\#O-(2d-1)\#Z} \text{ for all } n,
\]
where $\#Z = \lim_{n \to \infty} \#Z_n$ and $\#O = \lim_{n \to \infty} \#O_n$. Putting this in \eqref{eq: the_goal}, we obtain
\begin{equation}\label{eq: the_new_goal}
\text{if }  \EE q^{-\#O-(2d-1)\#Z} < \infty,~ \text{then } p \leq p_{shield}.
\end{equation}

We compute the expectation in the following lemma. Along with \eqref{eq: the_new_goal}, it immediately implies the main result, Theorem~\ref{lowerbound of shield}. (The condition $p_2<1$ for $d \geq 4$ will be verified in Lemma~\ref{lem: p_2_p_d}.)
\end{proof}

\begin{lemma}\label{lem: joint_mgf}
Assume that $p_2 < 1$ and that $q=1-p$ satisfies
\[
dq^{2d-1} > 1 \text{ and } f(q) < q,
\]
where
\[
f(q) = \left( \frac{1}{d} - \frac{1}{d^2}\right)\left( dq^{2d-1}-1\right)^{-1} + p_2 - \frac{1}{d^2}.
\]
Then
\[
\mathbb{E}q^{-\#O-(2d-1)\#Z} = (1-p_2) \left( 1- \frac{1}{d}\right)\frac{dq^{2d-1}}{dq^{2d-1}-1} (q- f(q))^{-1} < \infty.
\]
\end{lemma}
\begin{proof}
Let $h_n = ||S_n-S_n'||_1$ for $n \geq 0$, and note that $(S_n-S_n')_{n \geq 0}$ is a Markov chain on $\mathbb{Z}^d$ started at the origin. The sequence $(h_n)$ takes values in $\{0, 2, \dots\}$, but is not a Markov chain. However, computations give the following probabilities for it:
\begin{equation}
\begin{split} \label{eq: some probabilities of h_k}
& \PP(h_k=0 \mid h_{k-1}=0) = \frac{1}{d},\;\;\;\;\; \PP(h_k=2 \mid h_{k-1}=0) = 1-\frac{1}{d}, \text{ for } k \geq 1 \text{ and } \\
& \PP(h_k=2 \mid h_{k-1}=2) = \frac{3d-4}{d^2},\;\;\;\;\; \PP(h_k=0 \mid h_{k-1}=2) = \frac{1}{d^2} \text{ for } k \geq 2.\\
\end{split}
\end{equation}
Furthermore, since $p_2<1$, the strong Markov property implies $\#O<\infty$ a.s..

Let $(\mathcal{F}_n)$ be the filtration generated by $(X_k, X_k' : k= 1, \dots, n)$, and define the stopping times
\[
\begin{split}
&   \tau_0 =0, \;\;\;  \tau_1 = \inf \{n \geq 1 : h_n=2 \}, \text{ and generally}\\
& \tau_k = \inf \{n \geq \tau_{k-1}+1 : h_n=2\} \text{ for } k \geq 1.
\end{split}
\]
We then decompose the value of $\#Z$ according to ``excursions'' from the set $\{h_n = 2\}$. In other words, on the event $\{\#O = k\}$ for $k\geq 1$, we can write $\#Z = Z_1 + \cdots + Z_k$, where
\[
Z_i = \#\{n \in [\tau_{i-1}+1,\tau_i] : h_n = 0\}.
\]
(For this decomposition to hold, we need that $\#\{n \geq \tau_k+1 : h_n = 0\} = 0$. This holds a.s. on $\{\#O=k\}$, since after time $\tau_k$, the chain must move from $\{h_n=2\}$ to $\{h_n=4\}$, and never come back --- if it moves to $\{h_n=0\}$, it will a.s. move back to $\{h_n=2\}$ eventually by \eqref{eq: some probabilities of h_k}.)

Now we compute the expectation in the lemma iteratively, conditioning on each $\mathcal{F}_{\tau_k}$:
\begin{align}\label{eq: first equality of e^{alpha Z +beta O}}
&\EE q^{-\#O-(2d-1)\#Z} \nonumber \\
=~&\dsum_{k=1}^\infty \EE\left[q^{-k-(2d-1)(Z_1 + \dots + Z_k)} \mathbf{1}_{\{\#O=k\}}\right] \nonumber \\
=~&\dsum_{k=1}^\infty q^{-k} \EE\left[ \mathbb{E}\left[ q^{-(2d-1)(Z_1 + \dots + Z_k)} \mathbf{1}_{\{\tau_{k}<\infty, \tau_{k+1}=\infty\}} \mid \mathcal{F}_{\tau_k} \right] \right]  \nonumber\\
=~& \dsum_{k=1}^\infty \left( q^{-k} \EE \left[ q^{-(2d-1) (Z_1 + \cdots +Z_k)} \mathbf{1}_{\{\tau_k<\infty\} } \right]  \PP \left( \tau_{k+1}=\infty \mid \mathcal{F}_{\tau_k} \right) \right) \nonumber.
\end{align}
By the strong Markov property, $\mathbb{P}(\tau_{k+1} = \infty \mid \mathcal{F}_{\tau_k}) = \mathbb{P}(h_n \neq 2 \text{ for all } n \geq 2 \mid S_1-S_1' = x)$ for some (random) $x = S_{\tau_k}-S'_{\tau_k}$ in the set $\{z \in \mathbb{Z}^d : \|z\|_1 =2\}$. These $x$ are all of the form $\e_i-\e_j$ with $i\neq j$. By symmetry, these probabilities are the same for all $x$, and can be written as $\mathbb{P}(h_n \neq 2 \text{ for all } n \geq 2 \mid h_1=2) = 1-p_2$. (This argument is similar to the one that gives that $p_2<1$ implies $O<\infty$ a.s., stated below \eqref{eq: some probabilities of h_k}.) Therefore
\[
\mathbb{E} q^{-\#O-(2d-1)\#Z} = (1-p_2) \dsum_{k=1}^\infty q^{-k} \EE \left[ q^{-(2d-1) (Z_1 + \cdots +Z_k)} \mathbf{1}_{\{\tau_k<\infty\} } \right].
\]
Now conditioning on $\mathcal{F}_{\tau_{k-1}}$, this equals
\begin{equation}\label{eq: last_eqn}
(1-p_2) \dsum_{k=1}^\infty \left( q^{- k} \EE \left[  q^{-(2d-1) (Z_1+ \dots +Z_{k-1})} \mathbf{1}_{\{\tau_{k-1}<\infty\} } \EE \left[ q^{-(2d-1) Z_k} \mathbf{1}_{\{\tau_k < \infty\}} \mid \mathcal{F}_{\tau_{k-1}} \right]\right] \right).
\end{equation}
As before, by the strong Markov property, the term $\mathbb{E}\left[q^{-(2d-1)Z_k} \mathbf{1}_{\{\tau_k<\infty\}} \mid \mathcal{F}_{\tau_{k-1}}\right]$ for $k \geq 2$ is equal to $\mathbb{E}\left[ q^{-(2d-1)Z_2} \mathbf{1}_{\{\tau_2 < \infty\}} \mid S_1-S_1' = x\right]$ for some random $x = S_{\tau_{k-1}}-S'_{\tau_{k-1}}$ of the form $\e_i - \e_j$ for some $i \neq j$. These expectations are all the same by symmetry, so if we set
\[
f(q)= \EE \left[ q^{-(2d-1) Z_2} \mathbf{1}_{\{\tau_2 < \infty\}} \mid h_1=2\right],
\]
then \eqref{eq: last_eqn} gives us
\begin{align*}
&\mathbb{E}q^{-\#O-(2d-1)\#Z} \\
=~& (1-p_2) \left[ q^{-1} \mathbb{E}\left[ q^{-(2d-1)Z_1} \mathbf{1}_{\{\tau_0<\infty\}}\right] + \sum_{k=2}^\infty \left( q^{-k} f(q) \mathbb{E}\left[ q^{-(2d-1)(Z_1 + \dots + Z_{k-1})} \mathbf{1}_{\{\tau_{k-1}<\infty\}}\right] \right)\right].
\end{align*}
Last, we iterate the above procedure, conditioning successively on $\mathcal{F}_{\tau_{k-1}}, \mathcal{F}_{\tau_{k-2}}, \dots, \mathcal{F}_{\tau_{1}}$, to obtain
\[
\mathbb{E} q^{-\#O-(2d-1)\#Z} = (1-p_2) \EE \left[  q^{-(2d-1) Z_1} \mathbf{1}_{\{\tau_1 <\infty\} } \right]    \dsum_{k=1}^\infty \left( q^{- k} f(q)^{k-1} \right),
\]
or, because $\tau_1<\infty$ a.s. (see \eqref{eq: some probabilities of h_k}),
\begin{align}
\mathbb{E}q^{-\#O-(2d-1)\#Z}  &= (1-p_2) \EE q^{-(2d-1) Z_1}  \dsum_{k=1}^\infty \left( q^{-k} f(q)^{k-1} \right) \nonumber \\
&= (1-p_2) \mathbb{E}q^{-(2d-1)Z_1} (q- f(q))^{-1} \text{ if } f(q) < q.\label{eq: f(alpha) and expectation}
\end{align}

We now set out to compute the terms in \eqref{eq: f(alpha) and expectation}. Beginning with $f(q)$, because $h_2=0$ almost surely implies $\tau_2 < \infty$, we obtain
\begin{align}
f(q) &= \mathbb{E}\left[ q^{-(2d-1)Z_2} \mathbf{1}_{\{ \tau_2 < \infty, h_2 = 0\}} \mid h_1 = 2\right] + \mathbb{P}\left(  \tau_2 <\infty \text{ and } h_2 \neq 0 \mid h_1 = 2\right) \nonumber \\
&= \mathbb{E}\left[ q^{-(2d-1)Z_2} \mathbf{1}_{\{h_2=0\}} \mid h_1 = 2\right] + p_2 - \frac{1}{d^2}. \label{eq: f_q_middle}
\end{align}
Furthermore, using \eqref{eq: some probabilities of h_k}, the first term of \eqref{eq: f_q_middle} equals
\begin{align*}
&\frac{1}{d^2} \mathbb{E}\left[ q^{-(2d-1)Z_2} \mid h_1=2, h_2=0 \right] \\
=~& \frac{1}{d^2} \sum_{j=1}^\infty q^{-(2d-1)j} \mathbb{P}\left( h_2 = \dots = h_{j+1} = 0, h_{j+2} = 2 \mid h_1=2, h_2=0\right) \\
=~& \frac{1}{d^2} \sum_{j=1}^\infty q^{-(2d-1)j} \left( \frac{1}{d}\right)^{j-1} \left( 1-\frac{1}{d}\right) \\
=~& \frac{1}{d}\left( 1- \frac{1}{d} \right) \left(dq^{2d-1}-1\right)^{-1} \text{ if } dq^{2d-1} > 1.
\end{align*}
Putting this in \eqref{eq: f_q_middle}, we obtain
\begin{equation}\label{eq: new_f_q}
f(q) = \left( \frac{1}{d} - \frac{1}{d^2}\right)\left( dq^{2d-1}-1\right)^{-1} + p_2 - \frac{1}{d^2} \text{ when } dq^{2d-1} > 1.
\end{equation}

For the other term in \eqref{eq: f(alpha) and expectation}, we similarly compute when $dq^{2d-1} > 1$
\begin{align*}
\mathbb{E}q^{-(2d-1)Z_1} &= \sum_{j=0}^\infty q^{-(2d-1)j} \mathbb{P}(h_1 = \dots = h_j = 0, h_{j+1}=2) \\
&= \left( 1- \frac{1}{d}\right)\sum_{j=0}^\infty q^{-(2d-1)j} \left( \frac{1}{d}\right)^{j} \\
&= \left( 1- \frac{1}{d}\right)\frac{dq^{2d-1}}{dq^{2d-1}-1}.
\end{align*}
We place this and \eqref{eq: new_f_q} into \eqref{eq: f(alpha) and expectation} to complete the proof.
\end{proof}

\section{Proofs of Corollaries~\ref{corollary p_{shield}} and \ref{p_{shield >p_c} by numerical}}\label{sec: proofs_of_corollaries}

We will use the following result in the proofs of both corollaries.
\begin{lemma}\label{lem: p_2_p_d}
For $d \geq 4$, one has $p_2<1$. Furthermore, if we define
\[
p_d = \mathbb{P}(S_n = S_n' \text{ for some } n \geq 1),
\]
then
\begin{enumerate}
\item $p_2 = \frac{(d^2+1)p_d-d-1}{d^2p_d-d}$, and
\item 
\begin{align*}
p_d &\leq \frac{1}{d} + \left( 1- \frac{1}{d} \right) \frac{1}{d^2} + \frac{1}{d^2}\left( 1- \frac{1}{d}\right)\left( \frac{3d-4}{d^2}\right) \\
&+ \frac{1}{d^2} \left( \left( \frac{3d-4}{d^2}\right)^2 + \left( \frac{d^2-3d+3}{d^2}\right)\left(\frac{4}{d^2}\right)\right) \\
&+\sum_{k=5}^d \frac{k!}{d^k} + \sum_{j=1}^\infty \left(\frac{1}{d}\right)^{jd-1}\frac{(jd)!}{(j!)^d}.
\end{align*}
\end{enumerate}
\end{lemma}
\begin{proof}
We begin with item 1. We continue with the sequence $(h_n)$ from the previous section, where $h_n = \|S_n-S_n'\|_1$. As before, let
\[
\#Z_n = \#\{k=1, \dots, n: h_k=0\} \text{ and } \#O_n = \#\{k=1, \dots, n : h_k = 2\}.
\]
Then, recalling the probabilities in \eqref{eq: some probabilities of h_k}, we compute
\begin{align*}
\EE (1+\#Z_n)& = 1+ \dsum_{k=1}^n \PP(h_k=0) \\
&= 1 + \dsum_{k=1}^n \bigg( \PP(h_k=0, h_{k-1}=0) +  \PP(h_k=0, h_{k-1} =2 )\bigg)\\
& = 1 + \dsum_{k=1}^n \bigg( \frac{1}{d} \PP(h_{k-1}=0) +\frac{1}{d^2} \PP(h_{k-1}=2) \bigg)\\
&=1+\frac{1}{d} \EE (1+\#Z_{n-1}) + \frac{1}{d^2} \EE \#O_{n-1}.
\end{align*}
By the monotone convergence theorem, for $\#Z = \lim_{n \to \infty} \#Z_n$ and $\#O = \lim_{n \to \infty} \#O_n$, we have
\begin{equation} \label{expectation of Z and O}
\EE (1+\#Z) =  1+\frac{1}{d} \EE (1+\#Z)  + \frac{1}{d^2} \EE \#O.
\end{equation}

To write \eqref{expectation of Z and O} in terms of $p_2$ and $p_d$, we note that by the strong Markov property,
\begin{equation}\label{eq: p_d_1}
\mathbb{P}(\#Z=k) = p_d^k (1-p_d) \text{ for } k \geq 0, \text{ and }
\end{equation}
\begin{align}
\mathbb{P}(\#O=k) &= (1-p_2)p_2^{k-1} \mathbb{P}(h_k=2 \text{ for some } k \geq 1)  \nonumber \\
&= (1-p_2)p_2^{k-1} \text{ for } k \geq 1 \label{eq: p_2_1}.
\end{align}
Therefore 
\[
\mathbb{E}\#Z = \frac{p_d}{1-p_d} \text{ and } \mathbb{E}\#O = \frac{1}{1-p_2},
\]
and \eqref{expectation of Z and O} becomes
\[
\frac{1}{1-p_d} = 1 + \frac{1}{d(1-p_d)} + \frac{1}{d^2(1-p_2)}.
\]
This implies the first item of the lemma.

For the second item, we define the stopping time
\[
\tau =\inf \{k \geq 1 : h_k = 0\},
\]
so that $p_d = \PP( \tau < \infty)$. By a straightforward calculation,
\begin{equation}\label{eq: tau_upper_0}
\begin{split}
 \PP(\tau=1) &= \frac{1}{d}, \text{ and }\\
 \PP(\tau =2) &= \mathbb{P}(h_2=0\mid h_1 = 2) \mathbb{P}(h_1=2) \\
&=\bigg(1-\frac{1}{d} \bigg)\frac{1}{d^2}.
\end{split}
\end{equation}
We will need to compute both $\mathbb{P}(\tau=3)$ and $\mathbb{P}(\tau=4)$, and these are a little more complicated. We first claim that
\begin{equation}\label{eq: tau_upper_3}
\mathbb{P}(\tau=3) = \frac{1}{d^2}\left( \frac{3d-4}{d^2}\right)\left( 1-\frac{1}{d} \right).
\end{equation}
To show this use \eqref{eq: some probabilities of h_k} to write
\begin{align*}
\mathbb{P}(\tau=3) &= \mathbb{P}(h_1=2)\mathbb{P}(h_2=2 \mid h_1=2) \mathbb{P}(h_3=0 \mid h_1=2,h_2=2) \\
&= \left( 1- \frac{1}{d} \right) \frac{3d-4}{d^2} \mathbb{P}(h_3=0 \mid h_1=2,h_2=2).
\end{align*}
The last probability is written using the Markov property at time 2 as 
\[
\frac{\mathbb{E}\left[ \mathbb{P}(h_3=0 \mid \mathcal{F}_2) \mathbf{1}_{\{h_1=2,h_2=2\}}\right]}{\mathbb{P}(h_1=2,h_2=2)} = \frac{\mathbb{E}\left[ \mathbb{P}(h_2=0 \mid S_1-S_1'=x) \mathbf{1}_{\{h_1=2,h_2=2\}}\right]}{\mathbb{P}(h_1=2,h_2=2)},
\]
where $x$ is the (random) value of $S_2-S_2'$, which must be of the form $\e_i-\e_j$ for some $i\neq j$. These probabilities are constant as $x$ varies, and are equal to $\frac{1}{d^2}$. Therefore we obtain $\mathbb{P}(h_3=0 \mid h_1=2, h_2=2) = \frac{1}{d^2}$, and this shows \eqref{eq: tau_upper_3}.

The situation with $\{\tau=4\}$ is somewhat worse than that for $\{\tau=3\}$, and the form is
\begin{equation}\label{eq: tau_upper_4}
\mathbb{P}(\tau=4) \leq \frac{1}{d^2}\left( \left( \frac{3d-4}{d^2}\right)^2 + \left( \frac{d^2-3d+3}{d^2} \right)\left( \frac{4}{d^2}\right)\right) \left( 1-\frac{1}{d} \right).
\end{equation}
The analysis splits into 2 cases:
\begin{enumerate}
\item $(h_0, \dots, h_4) = (0, 2, 2, 2, 0)$,
\item $(h_0, \dots, h_4) = (0,2,4,2,0)$.
\end{enumerate}
The first case is computed exactly as we did for $\{\tau=3\}$: we obtain the form
\begin{equation}\label{eq: tau_upper_4_case_1}
\left( 1- \frac{1}{d} \right) \left( \frac{3d-4}{d^2}\right)^2 \frac{1}{d^2}.
\end{equation}
For the second, we get
\[
\left( 1-\frac{1}{d} \right) \mathbb{P}(h_2=4 \mid h_1=2) \mathbb{P}(h_3=2 \mid h_1=2,h_2=4) \frac{1}{d^2}.
\]
By \eqref{eq: some probabilities of h_k}, 
\[
\mathbb{P}(h_2=4 \mid h_1=2) = 1- \frac{1}{d^2} - \frac{3d-4}{d^2} = \frac{d^2-3d+3}{d^2},
\]
so we obtain
\begin{equation}\label{eq: tau_upper_4_case_2}
\left( 1- \frac{1}{d} \right) \frac{d^2-3d+3}{d^2} \mathbb{P}(h_3=2 \mid h_1=2, h_2=4) \frac{1}{d^2}.
\end{equation}
For the other term, we again use the Markov property to write it as
\[
\frac{\mathbb{E}\left[ \mathbb{P}(h_3=2 \mid S_2-S_2' = x) \mathbf{1}_{\{h_1=2, h_2=4\}}\right]}{\mathbb{P}(h_1=2,h_2=4)},
\]
where $x$ is the (random) value of $S_2-S_2'$. Up to symmetry, there are 3 different values of $x$:
\begin{enumerate}
\item[(A)] $2\e_i - 2\e_j$ for some $i\neq j$,
\item[(B)] $2\e_i - \e_j - \e_\ell$ for some distinct $i,j,\ell$, and
\item[(C)] $\e_i+\e_j-\e_\ell-\e_m$ for some distinct $i,j,\ell,m$.
\end{enumerate}
In case (A), $X_3$ must be $\e_j$ and $X_3'$ must be $\e_i$ to make $h_3=2$. This gives a probability of $\frac{1}{d^2}$. In case (B), $X_3$ must be $\e_j$ or $\e_\ell$ and $X_3'$ must be $e_i$, giving a probability of $\frac{2}{d^2}$. In case (C), $X_3$ must be $\e_\ell$ or $\e_m$ and $X_3'$ must be $\e_i$ or $\e_j$, giving a probability of $\frac{4}{d^2}$. In all cases, the probability is bounded above by $\frac{4}{d^2}$. Plugging this into \eqref{eq: tau_upper_4_case_2} gives an upper bound of
\[
\left( 1- \frac{1}{d} \right) \frac{d^2-3d+3}{d^2} \cdot \frac{4}{d^2} \cdot \frac{1}{d^2}.
\]
If we add this to \eqref{eq: tau_upper_4_case_1}, we obtain the claimed bound in \eqref{eq: tau_upper_4}.

For $\mathbb{P}(\tau=k)$ with $k \geq 5$, we use
\[
\PP(\tau=k) \leq \PP(S_k=S_k') \leq \max_{x \in H_k} \PP(S_k=x),
\]
where we recall that $H_k$ was defined in \eqref{eq: H_n_def}. Following \cite[p.~155]{MR684285}, for $1 \leq k \leq d$, the maximum above is attained when $x=\e_1 + \dots + \e_k$, so
\begin{equation}\label{eq: tau_upper_1}
\PP(\tau =k) \leq \max_{x \in H_k} \PP(S_k = x) \leq \frac{k!}{d^k} \;\;\text{for $1 \leq k \leq d$.}
\end{equation}

To bound $\mathbb{P}(\tau=k)$ for $k > d$, we first claim that $\max_{x \in H_j} \mathbb{P}(S_j=x)$ is nonincreasing in $j$. Indeed, if this were not true, then we could find $j$ such that $\max_{x \in H_j} \mathbb{P}(S_j = x) > \max_{y \in H_{j-1}} \mathbb{P}(S_{j-1} = y)$. Choosing $x$ corresponding to the maximum in $H_j$, we could compute
\begin{align*}
\mathbb{P}(S_j = x) = \sum_{y \in H_{j-1}} \mathbb{P}(S_{j-1} = y) \mathbb{P}(S_j=x \mid S_{j-1}=y) &< \mathbb{P}(S_j=x) \sum_{y \in H_{j-1}} \mathbb{P}(X_j = x-y) \\
&= \mathbb{P}(S_j=x),
\end{align*}
a contradiction. So, using the claim, if $k \geq jd$ for $j \geq 1$, we estimate, writing $y=(y_1, \dots, y_d)$,
\begin{align}
\PP(\tau=k) \leq \PP(S_k=S_k') \leq \max_{x \in H_k} \PP(S_k=x) \leq \max_{y \in H_{jd}} \PP(S_{jd}=y) &= \max_{y \in H_{jd}} \left[ \frac{1}{d^{jd}} \cdot \frac{(jd)!}{y_1! \cdot \dots \cdot y_d!}\right] \nonumber \\
&\leq \left( \frac{1}{d}\right)^{jd} \frac{(jd)!}{(j!)^d}. \label{eq: tau_upper_6}
\end{align}
If we write
\[
p_d = \mathbb{P}(\tau=1) + \mathbb{P}(\tau=2) + \mathbb{P}(\tau=3) + \mathbb{P}(\tau=4) + \sum_{k=5}^d \mathbb{P}(\tau=k) + \sum_{j=1}^\infty \sum_{k=jd+1}^{(j+1)d} \mathbb{P}(\tau =k),
\]
and use \eqref{eq: tau_upper_0} for the first and second terms, \eqref{eq: tau_upper_3} for the third, \eqref{eq: tau_upper_4} for the fourth, \eqref{eq: tau_upper_1} for the first sum, and \eqref{eq: tau_upper_6} for the last, we obtain the claimed inequality in item 2.

%

Finally, we show that for $d \geq 4$, one has $p_2 < 1$. It suffices, in fact, to show that $p_d<1$ since $p_2 = 1 - \frac{1-p_d}{d^2p_d-d}$, and if $p_d<1$ then the numerator is positive (the denominator is always positive since $p_d > \mathbb{P}(h_1=0) = \frac{1}{d}$). To show $p_d<1$, it is enough by \eqref{eq: p_d_1} to show that $\mathbb{E}\#Z<\infty$ and so we estimate as above, using Stirling's approximation with $ 1 \leq  \frac{n!}{n^n e^{-n} \sqrt{2\pi n}} \leq e^{\frac{1}{12n}} $ for all $n \geq 1$ from \cite[Eq.~(9.15)]{MR0228020} to obtain
\[
\left( \frac{1}{d} \right)^{jd} \frac{(jd)!}{(j!)^d} \leq \sqrt{2\pi d}\frac{e^{\frac{1}{12d}}}{ (2\pi )^{\frac{d}{2}}}  j^{\frac{1-d}{2}}.
\]
and
\begin{align*}
\mathbb{E}\#Z = \sum_{k=1}^\infty \mathbb{P}(h_k=0) \leq \sum_{k=1}^\infty \max_{x \in H_k} \mathbb{P}(S_k=x) &\leq \sum_{k=1}^d \frac{k!}{d^k} + \sum_{j=1}^\infty \sum_{k=jd+1}^{(j+1)d} \frac{(jd)!}{(j!)^d} \\
& \leq \sum_{k=1}^d \frac{k!}{d^k} + d\sqrt{2\pi d}\frac{e^{\frac{1}{12d}}}{ (2\pi )^{\frac{d}{2}}} \sum_{j=1}^\infty j^{\frac{1-d}{2}}.
\end{align*}
This is finite for $d \geq 4$.
\end{proof}

 \begin{proof} [Proof of Corollary \ref{corollary p_{shield}}]
 From consequence \eqref{eq: thm_1_consequence} of Theorem \ref{upperbound of shield},
\begin{equation*}
\begin{split}
\limsup_{d \rightarrow \infty} \frac{p_{shield}(d)}{\frac{\log d}{2d}} \leq 1.
\end{split}
\end{equation*}

For the lower bound, we put $p = p(d) = \frac{a \log d}{2d}$ for $a \in (0,1)$, and check that $p$ satisfies the conditions of Theorem~\ref{lowerbound of shield}. First, $p_2<1$ for large $d$ by Lemma~\ref{lem: p_2_p_d}. Next, one has
\begin{align*}
(1-p)^{2d-1} = \exp\left( (2d-1) \log \left( 1- \frac{a \log d}{2d}\right) \right) &= \exp\left( -a (1+o(1)) \log d \right) \\
&= d^{-a + o(1)} \text{ as } d \to \infty.
\end{align*}
This is $> \frac{1}{d}$ for large $d$, so the first assumption of Theorem~\ref{lowerbound of shield} holds. For the second, the calculation is similar: its left side equals
\[
\frac{1}{1-\frac{a \log d}{2d}} \left( p_2 - \frac{1}{d^2} + \frac{1}{d}\left( 1- \frac{1}{d}\right) ( d^{1-a+o(1)} - 1)^{-1} \right),
\]
which is $p_2(1+o(1))$ as $d \to \infty$. Since $p_2<1$, we see this the left side is $<1$ for large $d$, and this verifies item 2. In conclusion, we find that $p_{shield}(d) \geq \frac{a \log d}{2d}$ for large $d$, whenever $a<1$, and so
\[
\liminf_{d \to \infty} \frac{p_{shield}(d)}{\frac{\log d}{2d}} \geq 1.
\]
\end{proof}

\begin{proof} [Proof of Corollary~\ref{p_{shield >p_c} by numerical}] 
To find values of $d$ for which $p_{shield}(d) > p_c(d)$, we will need a useful upper bound for $p_c$. Unfortunately, we only have explicit upper bounds for the threshold of oriented percolation. We define the probability
\[
\rho_d = \PP( S_k=S_k', \; S_{k+1}=S_{k'+1} \text{ for some $k \geq 0$}), 
\]
and use \cite[Eq.~(1.1)]{MR684285}, which states that the oriented threshold satisfies $\vec{p}_c(d) \leq \rho_d$. Since $p_c(d) \leq \vec{p}_c(d)$, we obtain $p_c(d) \leq \rho_d$.

Define the stopping time $\hat \tau=\inf \{k \geq 0 : S_k =S_k', S_{k+1}=S_{k+1}'\}$, so that $\rho_d = \dsum_{k=0}^\infty \PP(\hat \tau = k)$. By similar calculations to those in the proof of Corollary~\ref{corollary p_{shield}} (the following are listed in \cite[p.~155-156]{MR684285}), 
\[
 \PP(\hat \tau =0) =\frac{1}{d}, \;\;\;  \PP(\hat \tau =1)=0,~\PP(\hat \tau=2)=\frac{1}{d^3}-\frac{1}{d^4},
\]
for $l = 1, \dots, d$,
\[
\mathbb{P}(l \leq \hat \tau \leq d) \leq \sum_{k=l}^d \frac{1}{d} \cdot \frac{k!}{d^k},
\]
and
\[
\mathbb{P}(d < \hat \tau < \infty) \leq  \sum_{j=1}^\infty \left( \frac{1}{d} \right)^{jd} \frac{(jd)!}{(j!)^d}.
\]

We will again want to separate out the cases $\hat \tau=k$ for $k=3,4$. Doing calculations similar to those in the proof of Lemma~\ref{lem: p_2_p_d}, we obtain
\begin{align*}
\mathbb{P}(\hat \tau = 3) &= \left( 1- \frac{1}{d} \right) \frac{3d-4}{d^2} \cdot \frac{1}{d^3} \\
\mathbb{P}(\hat \tau = 4) &\leq \frac{1}{d^3} \left[ \left( \frac{3d-4}{d^2} \right)^2 + \left( \frac{d^2-3d+3}{d^2} \right) \left( \frac{4}{d^2}\right) + \left( \frac{1}{d^2} \right) \left( 1- \frac{1}{d} \right) \right] \left( 1- \frac{1}{d} \right).
\end{align*}
(In the first case, the relevant $(h_n)$ vector is $(h_0, \dots, h_4) = (0,2,2,0,0)$ and for the second case, they are $(h_0, \dots, h_5) = (0,2,2,2,0,0)$, $(0,2,4,2,0,0)$, and $(0,2,0,2,0,0)$.)

Combining these estimates, we obtain
\begin{equation}\label{eq: p_c_upper_bound}
\begin{split}
p_c \leq \rho_d & \leq  \frac{1}{d} + \frac{1}{d^3}-\frac{1}{d^4} +  \frac{1}{d^3} \left( \frac{3d-4}{d^2} \right) \left( 1- \frac{1}{d} \right) \\
&+ \frac{1}{d^3} \left[ \left( \frac{3d-4}{d^2} \right)^2 + \left( \frac{d^2-3d+3}{d^2}\right) \left( \frac{4}{d^2} \right) + \left( \frac{1}{d^2} \right) \left( 1-\frac{1}{d} \right) \right] \left( 1-\frac{1}{d} \right) \\
&+\dsum_{k=5}^d \frac{k!}{d^{k+1}}  + \sum_{j=1}^\infty \left( \frac{1}{d}\right)^{jd} \frac{(jd)!}{(j!)^d}   =: g(d)
\end{split}
\end{equation}

To give an explicit lower bound on $p_{shield}(d)$, we will show that for $p = g(d)$, the two conditions of Theorem~\ref{lowerbound of shield} hold. That is, we will show that
\[
g(d) < 1 - \left( \frac{1}{d} \right)^{\frac{1}{2d-1}}
\]
and
\[
\frac{1}{1-g(d)} \left( p_2 - \frac{1}{d^2} + \frac{1}{d}\left( 1- \frac{1}{d} \right) (d (1-g(d))^{2d-1}-1)^{-1} \right) < 1.
\]
For any $d$ such that these inequalities hold, we must have $p_{shield}(d) > p_c(d)$. Indeed, since the left side of either inequality is a continuous function of $g(d)$, they will also hold for some number $\hat p > g(d)$ sufficiently close to $g(d)$, and we will have $p_c \leq g(d) < \hat p \leq p_{shield}(d)$.

To show the two inequalities above, we recall Lemma~\ref{lem: p_2_p_d} and the bounds contained therein. From there, we define
\begin{align*}
B(d) &= \frac{1}{d} + \left( 1- \frac{1}{d} \right) \frac{1}{d^2} + \frac{1}{d^2} \left( \frac{3d-4}{d^2} \right) \left( 1-\frac{1}{d} \right) \\
&+ \frac{1}{d^2} \left[ \left( \frac{3d-4}{d^2} \right)^2 + \left( \frac{d^2-3d+3}{d^2} \right)\left( \frac{4}{d^2} \right) \right] \left( 1-\frac{1}{d} \right) + \sum_{k=5}^d \frac{k!}{d^k} +  \sum_{j=1}^\infty \left( \frac{1}{d}\right)^{jd-1} \frac{(jd)!}{(j!)^d}
\end{align*}
and
\[
t(x) = \frac{(d^2+1)x-d-1}{d^2x-d}.
\]
(The function $t$ is defined so that $t(p_d) = p_2$.) Because $t(x) = 1- \frac{1-x}{d^2x-d}$, it is monotone nondecreasing for $x > 1/d$. Therefore, since $1/d < p_d \leq B(d)$, one has $p_2 \leq t(B(d))$, and we see that it will suffice to show that
\begin{equation}\label{eq: to_check_1}
g(d) \left( 1 - \left( \frac{1}{d}\right)^{\frac{1}{2d-1}} \right)^{-1} < 1
\end{equation}
and
\begin{equation}\label{eq: to_check_2}
\frac{1}{1-g(d)} \left( t(B(d))  - \frac{1}{d^2} + \frac{1}{d} \left( 1- \frac{1}{d} \right) (d (1-g(d))^{2d-1}-1)^{-1}\right) < 1.
\end{equation}
Table~\ref{table: value of left sides} shows computed values of the left sides of these inequalities. Their values drop below 1 between dimensions 9 and 10. 
\end{proof}

\bigskip
\noindent
{\bf Acknowledgements.} The research of M. D. is supported by an NSF CAREER grant. The research of C. M. N. is supported by NSF grant DMS-1507019. The authors thank an anonymous referee of a previous version of the paper for their helpful comments and suggestions, in particular for pointing out the relation to frozen percolation models, and for suggesting the improved bound in inequality \eqref{eq: new_edge_bound}.

\begin{table}[h]
\caption{The values of the left sides of \eqref{eq: to_check_1} and \eqref{eq: to_check_2}.  The maximum of the two values drops below 1 between $d=9$ and 10. Because both inequalities hold for $10 \leq d \leq 18$, one has $p_c(d) < p_{shield}(d)$ for these $d$. (Values computed using Mathematica. Those for $d \geq 15$ were computed by first applying Stirling's approximation in the definitions of $g(d)$ and $B(d)$ to save computer time.) \medskip}   \label{table: value of left sides}
\centering 
\begin{tabular}{l |  c   c } 
\hline\hline 
 $d$ & LHS of \eqref{eq: to_check_1}  & LHS of \eqref{eq: to_check_2}
\\ [0.5ex]
\hline 
 9 &  0.953 734 5     &  1.545 555 
\\[-1ex]
\\

 10 & 0.897 595 0     & 0.943 856 
  \\[-1ex]
\\

 11  & 0.855 878 5    & 0.697 538 
  \\[-1ex]
\\
12  & 0.822 865 5  &  0.545 351 
 \\[-1ex]
\\
13 &  0.795 549 3  & 0.443 074 
\\[-1ex]
\\
14 & 0.772 244 9 &  0.371 047  
  \\[-1ex]
\\
15 & 0.751 938 7  &   0.337 635 
\\[-1ex]
\\
16 &  0.733 976 5     &    0.293 250      
  \\[-1ex]
\\
17  &   0.717 908 0  &    0.260 608   
\\[-1ex]
\\
18 &   0.703 406 0   & 0.235 671 
  \\[-1ex]
\\
\hline 
\end{tabular}
\end{table}

\appendix

\section{Numerical results}
If we use numerical values of $p_c$, the result can be reduced to $d=7$. In other words, we can show that $p_{shield}(d) > p_c(d)$ for $d \geq 7$. The second column of Table~\ref{numerical result} shows numerical values of $p_c  = p_c^{bond}$ for dimensions 5-9. The third column gives lower bounds for $p_{shield}(d)$ for these dimensions. The fourth gives the maximum of the left sides of \eqref{eq: to_check_1} and \eqref{eq: to_check_2} when setting $p$ equal to the value in the third column. Because this maximum is $<1$, it shows that the value in the second column is indeed a lower bound for $p_{shield}$. One can see that the lower bound for $p_{shield}$ is larger than the value of $p_c$ for dimensions 7-9. 
%
%


\begin{table}[h] 
\caption{Numerical values of $p_c=p_c^{bond}$ and lower bounds for $p_{shield}$. The top numerical value of $p_c$ comes from \cite{MR1976824} and the bottom value comes from \cite{2018arXiv180608067M}. The fourth column is the maximum of the left sides of the first and second conditions in Theorem~\ref{lowerbound of shield} when $p$ is set equal to the lower bound for $p_{shield}$, which is the value in the third column. The value in the third column increases above that in the second between $d=6$ and 7. (Data computed using Mathematica.) \bigskip}     \label{numerical result}
\centering 
\begin{tabular}{l c  c c } 
\hline\hline 
 $d$ & $p_c^{bond}$  & lower bound of $p_{shield}$ & max. of left sides in  Thm.~\ref{lowerbound of shield}
\\ [0.5ex]
\hline 

&  0.118 171 8 &   \\[-1ex]

\raisebox{1.5ex}{5} & \raisebox{0.5ex}{ 0.118 171 5 }       & \raisebox{1.5ex}{0.020 681 5 } &   \raisebox{1.5ex}{0.999 99} \\[1ex]


&  0.094 201 9 &   \\[-1ex]
\raisebox{1.5ex}{6} & \raisebox{0.5ex}{  0.094 201 6 }       & \raisebox{1.5ex}{0.053 237 0} &   \raisebox{1.5ex}{0.999 99} \\[1ex]


&  0.078 675 2  &     \\[-1ex]
\raisebox{1.5ex}{7} & \raisebox{0.5ex}{ 0.078 675 2 }      & \raisebox{1.5ex}{0.081 242 1} &   \raisebox{1.5ex}{0.999 99} \\[1ex]


&  0.067 708 3 &     \\[-1ex]
\raisebox{1.5ex}{8} & \raisebox{0.5ex}{  0.067 708 4 }     & \raisebox{1.5ex}{0.098 080 4}  &   \raisebox{1.5ex}{0.999 99} \\[1ex]


& 0.059 496 0   &     \\[-1ex]
\raisebox{1.5ex}{9} & \raisebox{0.5ex}{   0.059 496 0 }    & \raisebox{1.5ex}{0.103 788 9} &   \raisebox{1.5ex}{0.999 84}\\[1ex]



\hline 
\end{tabular}
\end{table}

\newpage

\bibliographystyle{amsplain}
\bibliography{ref}

\clearpage

\end{document}